\documentclass[11pt, reqno]{amsart}
\usepackage{euscript,amscd,amsgen,amsfonts,amssymb,latexsym}

\newcommand{\eqdef}{\stackrel{\scriptscriptstyle\rm def}{=}}

\newtheorem{theorem}{Theorem}

\newtheorem{proposition}{Proposition}

\newtheorem{lemma}{Lemma}
\newtheorem{remark}{Remark}
\newtheorem*{*remark}{Remark}

\newcommand{\beha}{\begin{enumerate}}
\newcommand{\behe}{\end{enumerate}}
\renewcommand{\epsilon}{\varepsilon}

\newcommand{\cM}{\EuScript{M}}

\newcommand{\bR}{{\mathbb R}}

\newcommand{\bN}{{\mathbb N}}

\newcommand{\cA}{{\mathcal A}}

\def\C{C\!\!\!\!I}                  \def\oc{\overline \C}

\def\1{1\!\!1}

\def\and{\text{ and }}

                        \def\^{\tilde}

\def\Per{{\rm Per}}
\def\Pr{{\rm Per_{rep}}}

\def\1{1\!\!1}

\DeclareMathSymbol{\varnothing}{\mathord}{AMSb}{"3F}
\renewcommand{\emptyset}{\varnothing}
\title[Topological pressure for holomorphic dynamical systems]{Topological
  pressure for one-dimensional holomorphic dynamical systems} 
\author{Katrin Gelfert}\address{IM PAN, 00-950 Warszawa, ul. \'Sniadeckich 8, P.O. Box 137}\email{gelfert@pks.mpg.de} 
\urladdr{http://www.pks.mpg.de/$\sim$gelfert/}

\author{Christian Wolf}\address{Department of Mathematics, Wichita State
  University, Wichita, KS 67260, USA}\email{cwolf@math.wichita.edu} 
\urladdr{http://www.math.wichita.edu/$\sim$cwolf/}

\begin{document}

\begin{abstract}
  For a class of one-dimensional holomorphic maps $f$ of the Riemann sphere we
  prove that for a wide class of potentials $\varphi$ the topological pressure
  is entirely determined by the values of $\varphi$ on
  the repelling periodic points of $f$. This is a version of a classical result
  of Bowen for  hyperbolic diffeomorphisms in the holomorphic
  non-uniformly hyperbolic setting. 
\end{abstract}
\thanks{The research of K.G. was supported by the grant EU SPADE2. She is grateful to IM PAN for the hospitality.}
\maketitle

\section{Introduction}
In this paper we study the topological pressure $P_{\rm top}(\varphi)= P_{\rm
  top}(f,\varphi)$ of a continuous potential $\varphi$ with respect to a
one-dimensional holomorphic dynamical system $f$. To simplify the exposition
we discuss in the introduction exclusively the  case when $f$ is a rational
map of the Riemann sphere and present our more general results later on. 
Let $f:\oc\to\oc$ be a rational map with degree $d\geq 2$, and let $J$ denote
the Julia set of $f$, i.e. the closure of the repelling periodic points of $f$
(see~\cite{CG} for details). We are interested in the topological pressure
with respect to the dynamical system $f|J$.

We denote by $\Per_n(f)$ the fixed points of $f^n$ in $J$ and by
$\Per(f)=\bigcup_n\Per_n(f)$ the periodic points of $f$ in $J$. Moreover, let
$\Pr(f)\subset \Per(f)$ denote the set of repelling periodic points of $f$.
Given $\alpha>0$, $0<c\leq 1$, and $n\in \bN$ we define 
\begin{equation}\begin{split}
    \Per_n(\alpha,c) 
    = \{z\in \Per_n(f):
    & \lvert (f^k)'(f^i(z))\rvert \geq c\exp(k\alpha)\\
    & \text{ for all }  k\in \bN\text{ and }0\le i\le n-1\}. 
\end{split}\end{equation}
Thus, if  $\alpha\geq \alpha'$, $c\geq c'$, then
\begin{equation}\label{ni}
\Per_n(\alpha,c) \subset \Per_n(\alpha',c')
\end{equation}
and
\begin{equation*}
\Pr(f)=\bigcup_{\alpha>0}\bigcup_{c>0}\bigcup_{n=1}^\infty
\Per_n(\alpha,c).
\end{equation*}

Let $\cM$ denote the set of all $f$-invariant Borel probability measures on
$J$ endowed with weak$*$ topology. This makes $\cM$
to a compact convex space. Moreover, let $\cM_{\rm E}\subset \cM$ be the subset
of ergodic measures.
For $\mu\in \cM_E$ we define the Lyapunov exponent of $\mu$ by
\begin{equation}\label{mane}
\chi(\mu)= \int \log \lvert f'\rvert d\mu.
\end{equation}
It follows from Birkhoff's Ergodic Theorem that the \emph{pointwise Lyapunov
  exponent} at $z$, which is defined by
\begin{equation}
\chi(z)=\lim_{n\to\infty}\frac{1}{n}\log \lvert(f^n)'(z)\rvert,
\end{equation}
exists for $\mu$-a.e. $z\in J$ and coincides (whenever it exists) with
$\chi(\mu)$. 
We say that a measure $\mu$ is \emph{hyperbolic} if $\chi(\mu)>0$. We denote by
$h_\mu(f)$ the measure-theoretic entropy of $f$ with respect to $\mu$,
see for example~\cite{Wal:81} for the definition. Moreover, we denote by
$P_{\rm top}(\varphi)$ the topological pressure of $\varphi$ with respect to
$f$, see Section~\ref{press} for the definition. 
For $\varphi\in C(J,\bR)$ we define 
\begin{equation}\label{defalpha}
\alpha(\varphi)= \sup\{\chi(\mu): \mu\in\cM_E\cap ES(\varphi)\},
\end{equation}
where $ES(\varphi)$ denotes the set of equilibrium states of $\varphi$,
i.e. the set of measures $\mu\in\cM$ satisfying $P_{\rm top}(\varphi) =
h_\mu(f)+\int\varphi d\mu$ 
\footnote{Note that the supremum
  in~\eqref{defalpha} is in fact a maximum. This follows from the fact that
  $ES(\varphi)$ is a non-empty compact convex set whose extremal points are
  precisely the ergodic measures.}. 
We note that it follows from a general result of
Newhouse~\cite{N} (or alternatively from a theorem of Lyubich \cite{Lyu:83} or from Freire et al.~\cite{FreLopMan:83} in the case of rational maps) that for all  $\varphi\in C(J,\bR)$ we have
$ES(\varphi)\cap \cM_E\not=\emptyset$. 
Our main goal in this paper is to prove the following result (for a more
general case of not necessarily rational maps see Theorem~\ref{Main1} below): 

\begin{theorem}\label{Main}
 Let $f:\oc\to\oc$ be a rational map of and let $\varphi\in C(J,\bR)$ be a 
 H\"older continuous potential. 
 \begin{enumerate}
 \item [a)] If $\alpha(\varphi)>0$ then for all $0<\alpha<\alpha(\varphi)$
   we have 
   \begin{equation}\label{presspmain0}
     P_{\rm top}(\varphi) = 
     \lim_{c\to 0}\limsup_{n\to\infty}
     \frac{1}{n}\log \left(\sum_{z\in \Per_n(\alpha,c)} 
       \exp \left(\sum_{k=0}^{n-1}\varphi(f^k(z))\right)\right).
   \end{equation}
 \item [b)] If~\eqref{presspmain0} is true for some $\alpha>0$ then there
   exists an ergodic equilibrium state $\mu$ of $\varphi$ with $\chi(\mu)\geq
   \alpha$. 
 \end{enumerate}
\end{theorem}

We note that Theorem~\ref{Main} generalizes a well-known result of Bowen for
Axiom A diffeomorphisms to the case of holomorphic non-uniformly expanding dynamical
systems. For a related result in the case of non-uniformly hyperbolic
diffeomorphisms we refer to~\cite{GW}. 

We briefly mention work where related assumptions on the potentials have been
used. 
Note that the assumption of Theorem~\ref{Main} is satisfied if
$P_{\rm top}(\varphi)>\max_{z\in J}\varphi(z)$ and in particular if
$\max_{z\in J}\varphi(z)-\min_{z\in J}\varphi(z)<h_{\rm top}(f|J)$ are
satisfied, which are (much stronger than $\alpha(\varphi)>0$) open conditions
in the $C^0$ topology. 
The latter condition has been mentioned first in~\cite{HofKel:82} in the
context of piecewise monotonic maps of the unit interval and of a bounded
variation potential $\varphi$ to guarantee the existence and good ergodic
properties of equilibrium states for $\varphi$, using a spectral gap
approach. In~\cite{DenUrb:91}, it is shown that for a rational map of degree $\ge 2$ on the Riemannian sphere for a H\"older continuous potential $\varphi$ satisfying $P_{\rm top}(\varphi)>\sup\varphi$, there is a unique equilibrium state for $\varphi$. 
Analogous results are obtained for a class of non-uniformly
expanding local diffeomorphisms and H\"older continuous potentials satisfying
such a low oscillation condition (see~\cite{ArbMat:w06} and references 
therein).  

Przytycki et al.~\cite{PRS} consider a  pressure of the potential $-t\log\lvert f'\rvert$ which is defined as in~\eqref{presspmain0} except that they use \emph{all} periodic points rather than only points in $\Per_n(\alpha,c)$.
 They prove the equality between this pressure and various other types of
pressures in the case of  rational maps satisfying an additional
hypothesis that not too many periodic orbits with Lyapunov exponent close to 1 
move close together (which is satisfied if $f$ is a topological Collet-Eckmann map
or, equivalently, if $f$ is uniformly expanding on periodic orbits).
It would be interesting to know under which conditions their pressure coincides with
the pressure in~\eqref{presspmain0}.

This paper is organized as follows. In Section~\ref{sec2} we introduce a class
of one-dimensional holomorphic (not necessarily rational) dynamical systems
and discuss various notions of topological pressure. In  Section~\ref{sec3} we
prove our main result showing that for this class of systems the topological
pressure is entirely determined by the values of the potential on the
repelling periodic points.  

\section{Preliminaries}\label{sec2}

\subsection{A class of one-dimensional holomorphic dynamical systems}

Let $X \subset \oc$ be compact and let $f:X\to X$ be continuous. We say that
$f\in \cA(X)$ if there is an open neighborhood $U$ of $X$ such that $f$
extends to a holomorphic map on $U$ and for every $z\in U\backslash X$ 
\begin{equation}\label{cond}
\begin{split}
&\text{either } z \text{ leaves } U \text{ under iteration of } f,\\
&\text{or }\liminf_{n\to \infty}\frac{1}{n}\log\lvert (f^n)'(z)\rvert=0 
\end{split}
\end{equation}
holds. Without further specification we will always use a specific set $U$
associated with $X$ and $f$ and we will also denote the extension of $f$ to
$U$ by $f$.  
We note that in the particular case when $f$ is a rational map on $\oc$ with
Julia set $J$ then a normal family argument shows that $f\in \cA(J)$. 
For $f\in \cA(X)$ we will continue to use the  notation from Section 1
(e.g. $\Per(f)$, $\Pr(f)$, $\Per_n(\alpha,c)$, $\cM$, $\cM_E$, $\chi(\mu)$,
$\chi(z)$, $\alpha(\varphi)$, etc.) for $f|X$. 

Let now $U\subset \oc$ be open and $f:U\to \oc$ be holomorphic. 
We say that $f$ is \emph{expanding} on a compact $f$-invariant set
$\Lambda\subset U$ if there exist constants $c>0$ and $\beta>1$ such that 
\[
\lvert (f^n)'(z)\rvert \ge c\beta^n
\]
for all $n\in\bN$ and all $z\in \Lambda$.
We note that for $f\in \cA(X)$ every invariant expanding set $\Lambda\subset U$ is
contained in $X$. This follows from~\eqref{cond}.

\subsection{Various pressures}\label{press}

We first recall the definition of the classical topological pressure. Let
$(X,d)$ be a compact metric space and let $f\colon
X\to X$ be a continuous map. For $n \in {\mathbb N}$ we
define a new metric $d_n$ on $X$ by $d_n(z,y) = \max_{k=0,\ldots ,n-1}
d(f^k(z),f^k(y))$. A set of points $\{ z_i\colon i\in I \}\subset X$ is called
\emph{$(n,\varepsilon)$-separated} (with respect to $f$) if
$d_n(z_i,z_j)> \varepsilon$ holds for all $z_i,z_j$ with $z_i \ne
z_j$. Fix for all $\varepsilon>0$ and all $n\in\bN$ a maximal
(with respect to the inclusion) $(n,\varepsilon)$-separated set
$F_n(\epsilon)$. The \emph{topological pressure} (with respect to $f|X$)
is a map $ P_{\rm top}(f|X,.)\colon C(X,\bR)\to \bR$ defined by 
\begin{equation}\label{defdru}
  P_{\rm top}(f|X,\varphi) 
  = \lim_{\varepsilon \to 0}\limsup_{n\to \infty}
    \frac{1}{n} \log \left(\sum_{z\in F_n(\epsilon)}
      \exp S_n\varphi(z) \right),
\end{equation}
where
\begin{equation}\label{eqsn}
S_n\varphi(z)\eqdef\sum_{k=0}^{n-1}\varphi(f^k(z)).
\end{equation}
The \emph{topological entropy} of $f$ is defined by $h_{\rm top}(f|X)=P_{\rm
  top}(f|X,0)$.   
For simplicity we write $P_{\rm top}(\varphi)$ if there is no confusion about
  $f$ and $X$. 
Note that the definition of $P_{\rm top}(\varphi)$ does not depend on the
choice of the sets $F_n(\epsilon)$ (see~\cite{Wal:81}).
The topological pressure satisfies the following variational principle:
\begin{equation}\label{eqvarpri}
P_{\rm top}(\varphi)
= \sup_{\nu\in \cM} \left(h_\nu(f)+\int_\Lambda \varphi\,d\nu\right).
\end{equation}
Furthermore, the supremum in~\eqref{eqvarpri} can be replaced by
the supremum taken only over all $\nu\in\cM_{\rm E}$. We denote by
$ES(\varphi)$ the set of equilibrium states for $\varphi$, that is, the set of
measures attaining the supremum in~\eqref{eqvarpri}. We note that in general
$ES(\varphi)$ can be empty, however if $f\in \cA(X)$, then $ES(\varphi)$ contains
at least one (ergodic) measure. This follows from a result of
Newhouse~\cite{N}.  

Next we introduce a pressure-like quantity by using the values of
$\varphi$ on the periodic points in $X$. Let $\varphi\in C(X,\bR)$ and let
$0<\alpha$, $0<c\leq 1$. We define
\begin{equation*}
 Q_{\rm P}(\varphi,\alpha,c,n) =
\sum_{z\in \Per_n(\alpha,c)} \exp S_n\varphi(z)
\end{equation*}
if $\Per_n(\alpha,c)\ne\emptyset$ and
\begin{equation*}
Q_{\rm P}(\varphi,\alpha,c,n) =
\exp\left(n\min_{z\in X} \varphi(z)\right)
\end{equation*}
otherwise. Furthermore, we define
\begin{equation*}
P_{\rm P}(\varphi,\alpha,c) = \limsup_{n\to\infty}
                 \frac{1}{n}\log Q_{\rm P}(\varphi,\alpha,c,n).
\end{equation*}
It follows from the  definition that if $\Per_n(\alpha,c)\ne\emptyset$ for
some $n\in\bN$ then this is true already for infinitely many
$n\in\bN$. Therefore, in the case when $\Per_n(\alpha,c)\ne\emptyset$ for
some $n\in\bN$ then $P_{\rm P}(\varphi,\alpha,c)$ is
entirely determined by the values of $\varphi$ on
$\bigcup_{n\in\bN}\Per_n(\alpha,c)$.

\section{Pressure equals periodic point pressure}\label{sec3}

In this section we show for $f\in \cA(X)$ and a rather general class of
potentials that the topological pressure is entirely determined by the values
of the potential on the repelling periodic points. 
More precisely we prove the following theorem.

\begin{theorem}\label{Main1}
 Let $f\in \cA(X)$ and let $\varphi\in C(X,\bR)$ be a H\"older continuous
 potential with $\alpha(\varphi)>0$. 
 \begin{enumerate}
 \item [a)] If $\alpha(\varphi)>0$ then for all $0<\alpha<
   \alpha(\varphi)$ we have  
   \begin{equation}\label{presspmain}
     P_{\rm top}(\varphi) = \lim_{c\to 0}\limsup_{n\to\infty}\frac{1}{n}\log
     \left(\sum_{z\in \Per_n(\alpha,c)} 
       \exp  \left(\sum_{k=0}^{n-1}\varphi(f^k(z))\right)\right).
   \end{equation}
 \item [b)] If~\eqref{presspmain} is true for some $\alpha>0$ then there
   exists an ergodic equilibrium state $\mu$ of $\varphi$ with $\chi(\mu)\geq \alpha$.
 \end{enumerate}
\end{theorem}

\begin{*remark}{\rm
    Note that Theorem~\ref{Main1} immediately implies Theorem~\ref{Main}.}  
\end{*remark}

We delay the proof of Theorem~\ref{Main1} for a while and first prove some
preliminary results. 

\begin{lemma}\label{ha}
 Let $f\in \cA(X)$ and let $\Lambda$ be an invariant set on which $f$ is expanding. Let $\varphi\in
 C(\Lambda,\bR)$ be a H\"older continuous potential. Then  
 \begin{equation}\label{eqdruck}
   \limsup_{n\to\infty}\frac{1}{n}
   \log\left(\sum_{z\in\Per_n(f|\Lambda)}\exp S_n\varphi(z)\right) 
   \le P_{\rm top}(f|\Lambda,\varphi).
 \end{equation}
 In particular, if $f|\Lambda$ is topologically mixing, then we have equality
 in~\eqref{eqdruck}, and the limit superior is in fact a limit.
\end{lemma}

\begin{proof}
  Since $f|\Lambda$ is expanding it is expansive. Given any expansivity
  constant $\delta$, for every $n\in\bN$ and every $0<\varepsilon\le \delta$
  the set $\Per_n(f|\Lambda)$ is $(n,\varepsilon)$-separated. Now, 
  inequality~\eqref{eqdruck} follows from the fact that the
  definition~\eqref{defdru} can be replaced by the supremum taken over all 
  $(n,\epsilon)$-separated sets (see~\cite{Wal:81}). For the proof of the
  second statement we refer to  \cite[Chapter 7]{Rue:04}.
\end{proof}

\begin{remark}\label{Rem:1}{\rm
 The identity in~\eqref{eqdruck} holds in the more general case 
 of topologically mixing expanding maps (see~\cite[Chapter~7]{Rue:04}).
 In particular, if  $\Lambda$ is a repeller of a differentiable map $f$  such
 that $f|\Lambda$ is conjugate to a (one-sided) irreducible aperiodic subshift
 of finite type then \eqref{eqdruck} is an identity. 
}\end{remark}

\begin{proposition}\label{li}
  Let $f\in\cA(X)$ and let $\varphi\in C(X,\bR)$ be a  continuous
  potential. Then for all $\mu\in \cM_{\rm E}$ with $\chi(\mu)>0$ and for all
  $0<\alpha<\chi(\mu)$ we have 
  \begin{equation}\label{eqws}
  h_\mu(f)+\int\varphi d\mu \le
  \lim_{c\to 0} P_{\rm P}(\varphi,\alpha,c).
  \end{equation}
\end{proposition}

\begin{proof}
Consider $\mu\in \cM_{\rm E}$ with $\chi(\mu)>0$, and fix
$0<\alpha<\chi(\mu)$. Since $\chi(\mu)>0$, condition~\eqref{cond} implies that
${\rm supp}(\mu)\subset X$ and thus the left hand side of~\eqref{eqws} is
well-defined. 
It is a consequence of Katok's theory~\cite{K} in it's version
for holomorphic endomorphisms developed by Przytycki and Urbanski
(\cite[Chapter 9]{PrzUrb}, see  also~\cite{Prz:05}) that there exists a sequence $(\mu_n)_n$ of
measures $\mu_n\in \cM_{\rm E}$ supported on expanding sets $X_n\subset X$
such that 
\begin{equation}\label{holl}
  h_\mu(f)+\int_{}\varphi d\mu \le 
  \liminf_{n\to\infty}P_{\rm top}(f|X_n,\varphi)
\end{equation}
and $\mu_n\to\mu$ with respect to the weak$\ast$ topology.
Moreover, for each $n\in\bN$ there exist $m=m(n)\in \bN$ and $s=s(n)\in\bN$
such that $f^m| X_n$ is conjugate to the full shift in $s$ symbols. For every
$0<\varepsilon<\chi(\mu)-\alpha$ there is a number $n=n(\varepsilon)\in \bN$
such that
\begin{equation}\label{kuh}
h_\mu(f)+\int_{}\varphi d\mu -\varepsilon
\le P_{\rm top}(f|X_n,\varphi) .
\end{equation}
Moreover, there exists a number $c_0=c_0(n,\epsilon)$ with $0<c_0(n)\le1$ such
that for every periodic point $z\in X_n$ and every $k\in\bN$ we have
\begin{equation}\label{eq15}
  c_0^{-1}e^{k(\chi(\mu)-\varepsilon)} \le 
  \lvert (f^k)'(z)\rvert \le c_0e^{k(\chi(\mu)+\varepsilon)} 
\end{equation}
Note that~\eqref{eq15} is a consequence of the construction of the sets $X_n$
in~\cite[Chapter 9.6]{PrzUrb}.  
This implies that
\begin{equation}\label{kir}
  \Per_k(f)\cap X_n \subset \Per_k(\alpha,c_0)
\end{equation}
for all $k\in\bN$. 
Let $m$, $s\in\bN$ such that $f^m|X_n$ is topologically conjugate to the full
shift in $s$ symbols.
Since $mP_{\rm top}(f|X_n,\varphi) = P_{\rm top}(f^m|X_n,S_m\varphi)$
(see~\cite[Theorem 9.8]{Wal:81}), we can conclude that
\begin{equation*}
  h_\mu(f)+\int\varphi d\mu -\varepsilon
  \le \frac{1}{m}P_{\rm top}(f^m|X_n,S_m\varphi).
\end{equation*}
Recall that $S_m\varphi(z)=\sum_{i=0}^{m-1}\varphi(f^i(z))$.
It now follows from Remark~\ref{Rem:1} and an elementary calculation that
\begin{equation}\label{eqrep}
  \begin{split}
  h_\mu(f)&+\int \varphi d\mu -\varepsilon\\
    &\le
    \frac{1}{m}\lim_{k\to\infty}\frac{1}{k}\log\left(\sum_{z\in\Per_{mk}(f)\cap
    X_n} 
      \exp \left(\sum_{i=0}^{k-1}S_m\varphi(f^{im}(z))\right)\right)\\
  &\le \lim_{k\to\infty}\frac{1}{k}\log\left(\sum_{z\in\Per_k(f)\cap X_n}
    \exp S_{k}\varphi(z)\right).
\end{split}
\end{equation}
Combining~\eqref{kir} and~\eqref{eqrep} yields
\begin{equation*}
  h_\mu(f)+\int \varphi d\mu -\varepsilon
  \le \limsup_{k\to\infty}\frac{1}{k}\log
  \sum_{z\in\Per_k(\alpha,c_0)}\exp S_k\varphi(z).
\end{equation*}
Recall that by~\eqref{ni} the  map $c\mapsto P_{\rm P}(\varphi,\alpha,c)$ is
non-decreasing as $c\to 0^+$.
Since $\varepsilon>0$ is arbitrary the proof is complete.
\end{proof}

We are now in the situation to give the proof of Theorem \ref{Main1}.

\begin{proof}[Proof of Theorem~\ref{Main1}]
Let $0<\alpha$ and $0<c\leq 1$ such that $\Per_n(\alpha,c)\ne
\emptyset$ for some $n\in\bN$. 
We first prove that
\begin{equation}\label{le}
  P_{\rm P}(\varphi,\alpha,c) \le
  \sup_\nu \left\{h_\nu(f)+\int_\Lambda \varphi d\nu \right\}
  \leq P_{\rm top}(\varphi),
\end{equation}
where the supremum is taken over all $\nu\in \cM_{\rm E}$ with
$\alpha\le\chi(\nu)$. Note that the supremum in~\eqref{le} is not taken over
the empty set. 
The right hand side inequality in~\eqref{le} is a
consequence of the variational principle. 

In order to prove the left hand side inequality in~\eqref{le} we define
\[
\Lambda=\Lambda_{\alpha,c} 
\eqdef \overline{\bigcup_{n=1}^\infty \Per_n(\alpha,c)} .
\]
It follows from a continuity argument that $f$ is repelling on $\Lambda$.
Furthermore, for every $n\ge 1$ with $\Per_n(\alpha,c)\ne\emptyset$ we have, 
\begin{equation}\label{eqsi}
  \Per_n(f)\cap \Lambda = \Per_n(\alpha,c).
\end{equation}
Therefore, Lemma~\ref{ha} implies that
\begin{equation}\label{ghj}
P_{\rm P}(\varphi,\alpha,c)\leq P_{\rm top}(f|\Lambda,\varphi).
\end{equation}
It follows from  the variational principle that for every $\varepsilon>0$
there is a $\mu\in \cM_E$ which is supported in $\Lambda$ such that
\begin{equation}\label{hjk}
  P_{\rm top}(f|\Lambda,\varphi) -\varepsilon
  \le h_\mu(f) + \int \varphi d\mu
  \le P_{\rm top}(f|\Lambda,\varphi).
\end{equation}
Since $\mu$ is ergodic we have that $\chi(z)=\chi(\mu)$ for $\mu$-almost every
$z\in \Lambda$. It now follows from the continuity of $z\mapsto \lvert
f'(z)\rvert$ and the definition of $\Per_n(\alpha,c)$ that $\alpha\leq
\chi(z)$ for all $z\in \Lambda$. We conclude that
$\alpha\le\chi(\mu)$. Therefore, the left hand side inequality in~\eqref{le}
follows from~\eqref{ghj} and~\eqref{hjk}.  

Next, we prove that
\begin{equation}\label{pressp}
  P_{\rm top}(\varphi)\le \lim_{c\to 0}P_{\rm P}(\varphi,\alpha,c).
\end{equation}
Let $0<\alpha<\alpha(\varphi) $ and $0<\varepsilon<\alpha(\varphi)-\alpha$.
It follows from the definition of $\alpha(\varphi)$ (see
\eqref{defalpha}) that there exist  $\mu\in\cM_{\rm E}$ with 
$\chi(\mu)>\alpha(\varphi)-\varepsilon>\alpha$ such that  
\begin{equation}\label{gu}
  P_{\rm top}(\varphi)
  = h_{\mu}(f) + \int \varphi d\mu.
\end{equation}
Therefore,  Proposition~\ref{li} implies
\begin{equation}\label{eqwer}
  h_{\mu}(f)+\int\varphi d\mu
  \le \lim_{c\to 0}P_{\rm P}(\varphi,\alpha,c).
\end{equation}
Since $\varepsilon$ can be chosen arbitrary small,~\eqref{gu}
and~\eqref{eqwer} imply~\eqref{pressp}. 

Finally, we prove b). Let $\alpha>0$ such that~\eqref{presspmain} holds. 
For $n\ge 1$ and $c>0$ with $\Per_n(\alpha,c)\ne\emptyset$ we 
define the measure $\sigma_n=\sigma_n(\alpha,c,\varphi)\in\cM$ by 
\begin{equation}\label{tilsig}
\sigma_n=
\frac{1}{\sum\limits_{z\in\Per_n(\alpha,c)}\exp \left(S_n\varphi(z)\right)}
\sum_{z\in\Per_n(\alpha,c)}\exp \left(S_n\varphi(z)\right) \delta_z,
\end{equation}
where $\delta_z$ denotes the Dirac measure supported at $z$.
Note that every measure $\sigma_n=\sigma_n(\alpha,c,\varphi)$ defined
in~\eqref{tilsig} is in the convex hull  of the set
$\{\delta_z\colon z\in \Per_n(\alpha,c)\}$.  
Consider a subsequence $(\sigma_{n_k})_k$ converging to some measure
$\mu_{\alpha,c}=\mu_{\alpha,c}(\varphi)\in\cM$ in the weak$\ast$ topology. It
follows that $\chi(\mu_{\alpha,c})\ge \alpha$. 
Note that $f$ is expanding on $\Lambda_{\alpha,c}=\overline{\bigcup_{n=1}^\infty\Per_n(\alpha,c)}$. 
Thus, there exists an
expansivity constant $\delta=\delta(\alpha,c)$ for $f|\Lambda_{\alpha,c}$. In
particular, for every $n\in \bN$ and every  $0<\varepsilon\le\delta$ the set
$\Per_n(\alpha,c)$ is
$(n,\varepsilon)$-separated. As in the proof of \cite[Theorem 9.10]{Wal:81})
it follows that
\begin{equation}\label{krey}
  \limsup_{n\to\infty}\frac{1}{n}\log \sum_{x\in\Per_n(\alpha,c)}
  \exp\left(S_n\varphi(x)\right) \le 
  h_{\mu_{\alpha,c}}(f) + \int_X\varphi d\mu_{\alpha,c}.
\end{equation}
By construction, we have
\begin{equation}\label{ediu}
P_{\rm top}(\varphi) = 
\lim_{c\to 0}\left(h_{\mu_{\alpha,c}}(f) +
  \int_X\varphi d\mu_{\alpha,c}\right) .
\end{equation}
As $c$ decreases monotonically to zero, there exists a subsequence
$(\mu_{\alpha,c_k})_k$ converging to some measure $\mu=\mu(\varphi)\in\cM$
in the weak$\ast$ topology. Using the upper semi-continuity of the entropy map
and \eqref{ediu}, we can conclude that   
\[
\lim_{c_k\to 0}
\left(h_{\mu_{\alpha,c_k}}(f)+\int_X\varphi d\mu_{\alpha,c_k}\right)
= h_\mu(f) +\int_X\varphi d\mu
= P_{\rm top}(\varphi). 
\]
It remains to show that $\chi(\mu)\ge\alpha$. This is trivial in case the sets 
$\Lambda_{\alpha,c_k}$ do not accumulate at  critical points.
To handle the case that the
sets $\Lambda_{\alpha,c_k}$ possibly accumulate at a critical point $\gamma\in X$ we
consider a monotone decreasing sequence $(r_i)_i$ of positive 
numbers converging to $0$ and a monotone decreasing sequence of functions
$(\phi_i)_i$ in $C(X,\bR)$ such that the following holds:\\ 
(i) $\phi_i\geq \log|f'|$ and $\phi_i(z)=\log|f'(z)|$ for all $z\in X\setminus
B(\gamma,r_i)$.\\ 
(ii) $\phi_i(\gamma)\leq -i$.\\
In particular, $\phi_i$ converges pointwise to $\log|f'|$. Fix $i\in \bN$.
Since $\mu_{\alpha,c_k}$ converges to $\mu$ in the weak$\ast$ topology, we can
conclude that 
\begin{equation}
\int_X \phi_i d\mu = \lim_{k\to\infty} \int_X\phi_i d\mu_{\alpha,c_k}\geq
\liminf_{k\to\infty}\chi(\mu_{\alpha,c_k})\geq \alpha.  
\end{equation}
It now follows from~\eqref{mane} and the Monotone Convergence Theorem
that  
\[
\chi(\mu)=\lim_{i\to\infty} \int_X \phi_i d\mu \geq \alpha.
\]
One can choose $\mu$ to be ergodic by using an ergodic decomposition argument. 
The case when the sets $\Lambda_{\alpha,c_k}$ accumulate at finitely many
critical points can be treated entirely analogous. 
\end{proof}

\begin{*remark}{\rm 
    We note that we have used in the proof of Theorem~\ref{Main1} similar
    techniques as in our paper~\cite{GW} in the case of $C^2$-diffeomorphisms,
    as well as ideas from~\cite{ChuHir:03} where the topological entropy
    (i.e. $\varphi=0$) of surface diffeomorphisms is studied. 
}\end{*remark}

\end{document}